\documentclass{article}
\usepackage{amsfonts,amsmath,amsthm,mathrsfs,calrsfs,url,hyperref}
\title{A heuristic for the distribution of point counts for random curves over a finite field}
\author{Jeffrey D. Achter, Daniel Erman, Kiran S. Kedlaya, \\ Melanie Matchett Wood, and David Zureick-Brown}

\newtheorem{theorem}{Theorem}
\newtheorem{conjecture}[theorem]{Conjecture}
\newtheorem{corollary}[theorem]{Corollary}
\newtheorem{heuristic}[theorem]{Heuristic}
\newtheorem{lemma}[theorem]{Lemma}

\DeclareMathOperator{\an}{an}
\DeclareMathOperator{\Aut}{Aut}

\DeclareMathOperator{\Prob}{Prob}
\DeclareMathOperator{\Spec}{Spec}
\DeclareMathOperator{\Tr}{Tr}
\DeclareMathOperator{\USp}{USp}
\DeclareMathOperator{\Frob}{Frob}
\newcommand{\FF}{\mathbb{F}}
\newcommand{\ZZ}{\mathbb{Z}}
\newcommand{\QQ}{\mathbb{Q}}
\newcommand{\CC}{\mathbb{C}}
\def\cet{{c,\mathrm{et}}}
\def\stable{{\mathrm{stable}}}
\def\nonstable{{\mathrm{unstable}}}
\newcommand{\floor}[1]{{\left\lfloor #1 \right\rfloor}}

\begin{document}
\maketitle

\begin{abstract}
How many rational points are there on a random algebraic curve of large genus $g$ over a given finite field $\FF_q$? We propose a heuristic for this question motivated by a (now proven) conjecture of Mumford on the cohomology of moduli spaces of curves; this heuristic suggests a Poisson distribution with mean $q+1+1/(q-1)$. We prove a weaker version of this statement in which $g$ and $q$ tend to infinity, with $q$ much larger than $g$.
\end{abstract}

\section{Introduction}

The purpose of this paper is to propose a heuristic answer to the following question: what is the distribution of the number of rational points on a random   algebraic curve over a fixed finite field $\FF_q$ as the genus goes to infinity?
This is a question that can be translated into a question about the number of $\FF_q$-points of the moduli space $M_{g,n}$ of curves of genus $g$ with $n$ marked points.  Our fundamental heuristic assumption is that, in the Grothendieck-Lefschetz trace formula to count $\FF_q$-points on $M_{g,n}$, only the tautological classes contribute to the main term in the limit; we prove that this assumption implies the distribution of points on a random curves goes to a Poisson distribution with mean $q+1+1/(q-1)$.
Moreover, one can make a more precise statement in a certain limit where $q$ and $g$ both tend to infinity, but $q$ grows significantly faster than $g$.
These predictions and results are in the spirit of the work of Ellenberg--Venkatesh--Westerland \cite{EllenbergVW:cohenLenstra} on the relationship between stable homology of Hurwitz spaces and Cohen-Lenstra heuristics; they are also in a sense reciprocal to the work of Faber--Pandharipande \cite{FaberPandharipande}, in which point counts on $M_{g,n}$ for small $g$ are used to study the tautological classes.

Before making these statements more precise, we describe some similar questions which have been studied and indicate how this question differs somewhat from these.
The distribution of the number of rational points on a random (smooth, projective, geometrically irreducible) algebraic curve of a given class over a given finite\footnote{The corresponding question over a number field is also central in arithmetic statistics, but has a rather different flavor. See \cite{ho} for a comprehensive survey.} field has become a fundamental theme in the nascent field of \emph{arithmetic statistics}. Some examples of classes for which this topic has been
studied previously include hyperelliptic curves \cite{kurlberg-rudnick}, cyclic trigonal curves \cite{bdfl1}, noncyclic trigonal curves \cite{wood}, cyclic $p$-gonal curves \cite{bdfl3, xiong-l}, superelliptic and cyclic $m$-gonal curves \cite{cwz}, abelian covers of the line \cite{xiong-abelian}, Artin-Schreier curves \cite{bdfls, entin, bdfl4}, smooth plane curves \cite{bdfl2}, complete intersections in a fixed projective space \cite{bucur-kedlaya}, and curves in a fixed Hirzebruch surface \cite{erman-wood}. In each of these cases, every curve $C$ in the family maps to a fixed base space $\phi\colon C \rightarrow P$ and the (asymptotic) distribution of points on a random $C$ is given by a sum of independent bounded random variables associated to the rational points of the base space.  For each $p\in P(\FF_q)$, the associated random variable is the number of rational points in $\phi^{-1} (p)$.

Of course, the most natural and interesting family of smooth, projective curves is the family of all such curves, but proving a result about the distribution of points in this family seems currently out of reach.
The class of arbitrary curves differs from the previously mentioned classes in several important ways. One is that the number of rational points on the varying curve is not \emph{a priori} bounded.
A prior example sharing this property is that of \cite{kurlberg-wigman}, who considered curves lying in a sequence of surfaces with unbounded point counts. In this case, the average number of points on the curves is unbounded, so one is forced to renormalize to get a limiting distribution with finite mean, which turns out to be Gaussian.

The second distinctive feature of the class of all curves, which separates it from both \cite{kurlberg-wigman} and most of the preceding examples, is that the moduli space is not rational or even unirational.  That is, an arbitrary curve cannot be specified uniformly in terms of a collection of parameters.  This makes even the ``denominator'' in the question, the total number of curves over $\FF_q$ of a fixed genus, extremely difficult to understand.  (See \cite{deJongKatz} for an upper bound.)

Finally, the lack of nontrivial maps from curves in the family to a fixed space means there is no way to make sensible probabilistic models which split the point count into a sum of independent random variables.

Let us now make things more precise for the class of curves.
Let $M_g$ denote the fine moduli space of curves of genus $g$ in the sense of Deligne and Mumford \cite{deligne-mumford}; it is an object in the category of algebraic stacks over $\Spec(\mathbb{Z})$. The set $|M_g(\mathbb{F}_q)|$ of (isomorphism classes of) $\mathbb{F}_q$-rational points of $M_g$ may then be identified with the set of isomorphism classes of smooth, projective, geometrically connected curves of genus $g$ over $\mathbb{F}_q$.
To simplify notation, let us further identify $|M_g(\mathbb{F}_q)|$ with a set consisting of one curve in each isomorphism class.
For $C \in |M_g(\mathbb{F}_q)|$, let $\Aut(C)$ be the group of automorphisms of $C$ as a curve over $\mathbb{F}_q$ (not over an algebraic closure over $\mathbb{F}_q$).
We equip $|M_g(\mathbb{F}_q)|$ with the probability measure
in which each point $x$ is weighted proportionally to $1/\# \Aut(C)$.
This is well-understood to be the most natural way to count objects with automorphisms, and matches the weighting of points in the Lefschetz trace formula for Deligne-Mumford stacks given by Behrend \cite{behrend}.

Let $C_g$ be the (random) curve associated to a random $x \in |M_g(\mathbb{F}_q)|$ drawn according to the above probability measure.
For each $g,$ $\# C_g(\mathbb{F}_q)$ is a random variable taking values in the nonnegative integers, and we are interested in the limiting behavior of the distributions of these random variables as $g \to \infty$.  We prove that a heuristic assumption about the cohomology of $M_{g,n}$ (Heuristic~\ref{heur:tautological}) implies
 that these distributions converge to a Poisson distribution with mean $q+1+1/(q-1) = q + 1 + q^{-1} + q^{-2} + \cdots$; more precisely, we show that
Heuristic~\ref{heur:tautological} implies the following predictions.
\begin{conjecture} \label{conj:poisson1}
Put $\lambda := \lambda(q) =  q+1+1/(q-1)$.
\begin{enumerate}
\def\theenumi{\alph{enumi}}
\item
For all nonnegative integers $n$,
\[
\lim_{g \to \infty} \Prob(\#C(\mathbb{F}_q) = n: C \in |M_g(\mathbb{F}_q)|) = \frac{\lambda^n e^{-\lambda}}{n!}.
\]
\item
For all positive integers $n$,
\[
\lim_{g \to \infty} \mathbb{E}(\#C(\mathbb{F}_q)^n:  C \in |M_g(\mathbb{F}_q))|) = \sum_{i=1}^n \genfrac\{\}{0pt}{}{n}{i} \lambda^i,
\]
where $\genfrac\{\}{0pt}{}{n}{i}$ denotes a Stirling number of the second kind (i.e., the number of unordered partitions of $\{1,\dots,n\}$ into $i$ disjoint sets).
\end{enumerate}
\end{conjecture}
Note that part (b) implies part (a): the moment sequence of the Poisson distribution has exponential growth and thus determines the distribution uniquely \cite[Theorem~30.1]{billingsley}, and for such a limiting distribution convergence
at the level of moments implies convergence at the level of distributions
\cite[Theorem~30.2]{billingsley}.
If we let $(X)_n:=X(X-1)\cdots(X-n+1)$, then the falling moments
\begin{equation}
\lim_{g \to \infty} \mathbb{E}((\#C(\mathbb{F}_q))_n:  C \in |M_g(\mathbb{F}_q)|) = \lambda^n
\end{equation}
(for all positive integers $n$) are equivalent to the standard moments in (b) above.

Let $M_{g,n}$ denote the moduli space of curves of genus $g$ with $n$ distinct marked points, again as an algebraic stack over $\Spec(\ZZ)$.
Each element of $|M_{g,n}(\mathbb{F}_q)|$ may now be identified (by fixing a representative of each isomorphism class) with a tuple $(C, P_1,\dots,P_n)$ where $C$ is as before and
$P_1,\dots,P_n$ are distinct elements of $C(\mathbb{F}_q)$.
We equip the points of $|M_{g,n}(\mathbb{F}_q)|$ with the weights where $(C, P_1,\dots,P_n)$ has weight $1/\#\Aut(C, P_1,\dots,P_n)$ (i.e., we only consider automorphisms of $C$ fixing $P_1,\dots,P_n$). By an easy orbit counting argument,
$$
\mathbb{E}((\#C(\mathbb{F}_q))_n:  C \in |M_g(\mathbb{F}_q)|) = \frac{\# |M_{g,n}(\mathbb{F}_q)|}{\# |M_g(\mathbb{F}_q)|}
$$
(where $\#$ denotes weighted count).
Thus, we may rewrite Conjecture~\ref{conj:poisson1} as the statement
\begin{equation}\label{C:newvers}
\lim_{g \to \infty} \frac{\# |M_{g,n}(\mathbb{F}_q)|}{\# |M_g(\mathbb{F}_q)|} = \lambda^n.
\end{equation}
Let us now make explicit how we would like to study $\# |M_{g,n}(\mathbb{F}_q)|/\# |M_{g}(\mathbb{F}_q)|$ using the Grothendieck-Lefschetz-Behrend trace formula.
For a smooth Deligne-Mumford stack $X$ over $\FF_q$, the trace formula asserts that for any prime $\ell$ not dividing $q$,
\[
\#|X(\FF_q)| = \sum_{i=0}^{2 \dim(X)} (-1)^i \mathrm{Trace}(\mathrm{Frob}, H^i_{c,\mathrm{et}}(X_{\overline{\FF}_q}, \mathbb{Q}_{\ell}))
\]
where $X_{\overline{\FF}_q}$ denotes the base extension of $X$ from $\FF_q$ to $\overline{\FF}_q$, $H^i_{c,\mathrm{et}}(X_{\overline{\FF}_q}, \mathbb{Q}_{\ell})$ denotes compactly supported \'etale cohomology, and $\mathrm{Frob}$ is the geometric Frobenius automorphism on $X_{\overline{\FF}_q}$. By Deligne's proof of the Riemann hypothesis for algebraic varieties, each eigenvalue $\alpha$ of $\mathrm{Frob}$ on
$H^i_{c,\mathrm{et}}(X_{\overline{\FF}_q}, \mathbb{Q}_{\ell})$ is an algebraic integer with the property that for some $w \in \{0,\dots,i\}$ (called the \emph{weight} of $\alpha$), the conjugates of $\alpha$ in $\CC$ all have absolute value $q^{w/2}$.

This suggests that one should be able to estimate $\#|M_{g,n}(\FF_q)|$, and hence the ratio $\#|M_{g,n}(\FF_q)|/\#|M_g(\FF_q)|$, by computing the action of geometric Frobenius on the highest-degree cohomology groups of $M_{g,n,\overline{\FF}_q}$ and burying the other contributions to the trace formula in an error term.  Moreover, the highest degree cohomology groups with their Frobenius action are known exactly (see Theorem~\ref{T:etale}): they are spanned by so-called tautological classes (see below).
Unfortunately, this approach does not lead to any provable estimates for fixed $q$ because the Betti numbers of $M_{g,n,\overline{\FF}_q}$ grow superexponentially in $g$ (e.g., see \cite{harer-zagier} for the calculation of the Euler characteristic).  Thus, even though terms from lower degree cohomology groups contribute to the Grothendieck-Lefschetz sum with smaller weight, there are so many of them that they cannot \emph{a priori} be treated as negligible compared to the top-degree contributions.

Despite this imbalance, we can still make a reasonable heuristic about what we expect the asymptotics of the Grothendieck-Lefschetz sum to be.
One can classify the Frobenius eigenvalues of $H^*_{c,\mathrm{et}}$ of each weight $w$ as ``causal'' and ``random.''  The causal eigenvalues are the ones whose presence is compelled by the existence of certain algebraic cycles (in our case, the eigenvalues of the tautological classes); these eigenvalues must be integral powers of $q$. It is plausible to model the random eigenvalues of a given weight $w$ by the eigenvalues of a random unitary\footnote{In middle cohomology, it is more natural to use a random unitary symplectic matrix or a random Hermitian matrix instead, but the same discussion applies to these models.} matrix times $q^{w/2}$.
Let $d_{g,n}$ be the relative dimension of $M_{g,n}$ over $\Spec(\ZZ)$, which is $3g-3+n$ for $g>1$.
Let $b_k$ be the number of ``random'' eigenvalues of weight $2d_{g,n}-k$ (i.e., of \emph{coweight} $k$).
We have $b_k=0$ for $k\leq \frac{2g-2}{3}$; see Theorem~\ref{T:etale}.  For $k>\frac{2g-2}{3}$, if there are few eigenvalues of coweight $k$, e.g. $b_k=o(q^{k/2})$, then the weight $k$ eigenvalues contribute nothing to the Grothendieck-Lefschetz sum in the limit as $g\to\infty$.  On the other hand, if there are many eigenvalues of coweight $k$, and we model them with eigenvalues of a large random unitary matrix, we know from a result of Diaconis--Shahshahani
\cite{diaconis-shahshahani} that this matrix has bounded trace with high probability. It is thus a sensible heuristic to neglect the contribution of all but the causal eigenvalues.
Our neglect of the random Frobenius eigenvalues is also consistent with a commonly held philosophy in the study of moduli spaces, that no natural geometric questions depend on the non-tautological classes (e.g., see \cite{vakil}).

That this heuristic is sensible relies crucially on the fact that there are no random eigenvalues of large weight, which is a deep fact about the cohomology of moduli spaces of curves conjectured by Mumford and later proved using topological techniques (see Section~\ref{S:stability} for references).  The compactly supported \'etale cohomology in high degrees (or equivalently by Poincar\'e duality and a Betti-\'etale comparison isomorphism, the Betti cohomology in low degrees; see Theorem~\ref{T:etale}) is spanned by \emph{tautological} classes,
i.e., classes which arise from algebraic cycles produced by canonical morphisms between moduli spaces. The prototypical example of such a class is the first Chern class of the relative dualizing sheaf of the morphism $M_{g,n} \to M_{g,n-1}$ obtained by forgetting one marked point.

We may formalize our heuristic as follows.  Write $R^*_{c,\mathrm{et}}(M_{g,n,\overline{\FF}_q}, \mathbb{Q}_{\ell})$ for the subspace of $H^i_{c,\mathrm{et}}(M_{g,n,\overline{\FF}_q}, \mathbb{Q}_{\ell})$ generated by tautological classes, and put
$B^*_{c,\mathrm{et}}(M_{g,n,\overline{\FF}_q}, \mathbb{Q}_{\ell}):=H^*_{c,\mathrm{et}}(M_{g,n,\overline{\FF}_q}, \mathbb{Q}_{\ell})/R^*_{c,\mathrm{et}}(M_{g,n,\overline{\FF}_q}, \mathbb{Q}_{\ell})$.

\begin{heuristic}\label{heur:tautological}
As $g\to \infty$, only the tautological classes are asymptotically relevant to a Grothendieck-Lefschetz trace formula computation of $\#|M_{g,n}(\mathbb F_q)|$.  More precisely,
\[
\lim_{g\to\infty} \frac{\sum_{i=0}^{2 d_{g,n}- \frac{2g-2}{3} } (-1)^i \mathrm{Trace}(\mathrm{Frob}, B^i_{c,\mathrm{et}}(M_{g,n,\overline{\FF}_q}, \mathbb{Q}_{\ell}))}{q^{d_{g,n}}} =0.
\]
\end{heuristic}

It is convenient for our heuristic that the tautological classes are \emph{stable}.
This means that for $i \geq 2 d_{g,n} - \frac{2g-2}{3}$, the groups
$R^i_{c,\mathrm{et}}(M_{g,n,\overline{\FF}_q}, \mathbb{Q}_{\ell})$ (and thus  the groups
$H^i_{c,\mathrm{et}}(M_{g,n,\overline{\FF}_q}, \mathbb{Q}_{\ell})$) can be described in a manner independent of $g$, making it particularly nice to take the limit in $g$.  Further, the number of lower degree tautological classes is sufficiently bounded that we can ignore their contribution to the Grothendieck-Lefschetz sum.

Our first main result is the following theorem.
\begin{theorem}\label{T:getconj}
Heuristic~\ref{heur:tautological} implies Conjecture~\ref{conj:poisson1}.
\end{theorem}
Our second main result establishes unconditionally a weaker version of Conjecture~\ref{conj:poisson1} in which both $g$ and $q$ tend to infinity;
this result lends some credence to Conjecture~\ref{conj:poisson1}. In particular, since the error term is smaller than $q^{-m}$ for any fixed $m$, this result rules out any alternate conjecture in which each moment is a universal Laurent series in $q^{-1}$.

\begin{theorem}\label{T:weaker}
For any $K>144$, any function $q(g) > g^K$, and any nonnegative integer $n$, for $q = q(g)$
we have
\[
\lim_{g\rightarrow \infty}\frac{\# |M_{g,n}(\mathbb{F}_q)|}{\# |M_g(\mathbb{F}_q)|} = \lambda^n + O(q^{-g/6}).
\]
\end{theorem}

The key to proving Theorem~\ref{T:weaker} is that, so long as $q\gg g$,  the unstable homology is negligible in the Grothendieck-Lefschetz trace computation.

It would be interesting to compute what a heuristic similar to Heuristic~\ref{heur:tautological} suggests about the
average number of points on a stable curve of genus $g$, as $g\to \infty$.  Our approach fails to directly yield an answer.
In particular, we would seek a computation along the lines of
Lemma~\ref{L:Tasymptotic}, but this is complicated by the fact that the dimensions of the tautological cohomology
$R^i(\overline{M}_{g,n})$ can grow exponentially in $g$, as $g\to \infty$.

In Section~\ref{S:stability}, we review the topological results showing that the low degree singular cohomology of $M_{g,n}$ is tautological and giving a precise description of the cohomology groups.
In Section~\ref{S:stability etale}, we translate these results into compactly supported \'{e}tale cohomology using comparison isomorphisms and determine the effect of Frobenius.  In Section~\ref{S:heuristic implies conjecture} we prove Theorem~\ref{T:getconj}.  In Section~\ref{S:weaker}, we prove Theorem~\ref{T:weaker}.  In Section~\ref{S:random matrix models}, we outline some thoughts and questions about how a random matrix model might give evidence for or against Conjecture~\ref{conj:poisson1}.

\section*{Acknowledgements}
This project was started at the  AIM workshop ``Arithmetic statistics over number fields and function fields'' (January 27-31, 2014). Achter was supported by Simons Foundation grant 204164 and NSA grant H98230-14-1-0161.  Erman was supported by NSF grant DMS-1302057. Kedlaya was supported by NSF grant DMS-1101343; he also thanks MSRI for its hospitality during fall 2014 as supported by NSF grant DMS-0932078.  Wood was supported by the American Institute of Mathematics and NSF grants  DMS-1147782 and DMS-1301690.
Thanks also to Dan Abramovich, Jordan Ellenberg, Martin Olsson, Aaron Pixton, and Ravi Vakil for helpful discussions.

\section{Stability and tautological classes: singular cohomology}\label{S:stability}

Let $M_{g,\CC}^{\an}$ and $M_{g,n,\CC}^{\an}$ be the underlying topological spaces of the stacks $M_{g,\CC}$ and $M_{g,n,\CC}$.
We begin by recalling some deep results on the stable singular cohomology of $M_{g,n,\CC}^{\an}$. These results are typically stated without marked points; we must add a bit of extra analysis to deal with the markings.

\begin{theorem} \label{T:stable homology}
For any nonnegative integers $g,n,i$ with $i \leq \frac{2g-2}{3}$, there exists an isomorphism $H_i(M_{g,n,\CC}^{\an}, \QQ) \to H_i(M_{g+1,n,\CC}^{\an}, \QQ)$. By the universal coefficient theorem, this gives rise to an isomorphism
$H^i(M_{g,n,\CC}^{\an}, \QQ) \to H^i(M_{g+1,n,\CC}^{\an}, \QQ)$.
\end{theorem}
\begin{proof}
This was first proved with a slightly more restrictive bound on $i$ by
Harer \cite{Harer:stability-orientable, Harer:stability-spin}.
The statement as given includes results of several authors; see
\cite[Theorem 1.1]{Wahl:stability-handbook}.
\end{proof}

The proof of this result is ultimately topological: by Teichm\"uller theory, one may identify $M_{g,n,\CC}^{\an}$ up to homotopy with a classifying space of the mapping class group $\Gamma_{g,n}$ of a compact Riemann surface (without boundary) of genus $g$ with $n$ marked points. One may take a homotopy limit to obtain a group $\Gamma_{\infty,n}$ whose group (co)homology computes the stable (co)homology of $M_{g,n,\CC}^{\an}$.

Let us now momentarily restrict attention to the case $n=0$.
Following Mumford, we define the \emph{tautological ring} to be the graded polynomial ring $R := \QQ[\kappa_1, \kappa_2,\dots]$ with $\deg(\kappa_j) = 2j$.
We obtain a map from $R$ to the Chow ring of $M_g$ as follows: let $\psi$ be the relative dualizing sheaf of the morphism $M_{g,1} \to M_g$ which forgets the marked point, then let $\kappa_j$ be the pushforward of $\psi^{j+1}$ along $M_{g,1} \to M_g$.
\begin{theorem}
\label{T:stab-taut-class}
The induced map $R \to H^{*}(M_{g,\CC}^{\an}, \QQ)$ of graded rings is an isomorphism in degrees up to $\frac{2g-2}{3}$.
\end{theorem}
\begin{proof}
This follows from Theorem~\ref{T:stable homology} plus a theorem of Madsen and Weiss identifying $R$ with the stable cohomology ring \cite{madsen-weiss}. \end{proof}

We now consider the effect of marked points.
Define the tautological ring $R_n = R[\psi_1,\dots,\psi_n]$ with $\deg(\psi_i) = 2$. We obtain a map from $R_n$ to the Chow ring of $M_{g,n}$ as follows: map $\kappa_j$ as before, and map $\psi_i$ to the relative dualizing sheaf of the morphism $M_{g,n} \to M_{g,n-1}$ which forgets the $i$-th marked point.
\begin{theorem}
The induced map $R_n \to H^{*}(M_{g,n,\CC}^{\an}, \QQ)$ of graded rings is an isomorphism in degrees up to $\frac{2g-2}{3}$.
\end{theorem}
\begin{proof}
This follows from the existence of a homotopy equivalence
\[
B\Gamma_{\infty,n+1} \sim B\Gamma_{\infty,n} \times \CC \mathbb{P}^\infty
\]
as constructed in \cite[Corollary~1.2]{bodigheimer-tillmann} (see also
\cite[Theorem~4.3]{tillmann}).
\end{proof}

\section{Stability and tautological classes: \'etale cohomology}
\label{S:stability etale}

We next translate the stability of cohomology from singular cohomology to compactly supported \'etale cohomology, and determine the effect of Frobenius on the stable cohomology classes, in order to use the Grothendieck-Lefschetz-Behrend trace formula.

\begin{lemma} \label{L:comparison of cohomology}
Choose an embedding of $\overline{\QQ}_p$ into $\CC$.
Let $\overline{Y}$ be a smooth proper scheme over $\Spec(\ZZ_p)$.
Let $Z$ be a relative normal crossings divisor on $\overline{Y}$.
Let $G$ be a finite group acting on both $\overline{Y}$ and $Z$.
Put $Y = \overline{Y} - Z$ and let $X$ be the stack-theoretic quotient $[Y/G]$.
Then there are functorial isomorphisms
\[
H^{i}_{\mathrm{et}}(X_{\overline{\FF}_q}, \QQ_{\ell}) \cong H^{i}_{\mathrm{et}}(X_{\CC}, \QQ_{\ell}) \cong H^{i}(X_{\CC}^{\an}, \QQ_{\ell}).
\]
\end{lemma}
\begin{proof}
In case $G$ is trivial, the first isomorphism follows from \cite[Proposition~4.3]{nakayama} and the second isomorphism follows from \cite[Theorem I.11.6]{FreitagK:lectures-etale} (for more details, see
\cite[Proposition~7.5]{EllenbergVW:cohenLenstra}). The general case follows from this special case by applying
the Hochschild-Serre spectral sequence \cite[Theorem 2.20]{Milne:etaleBook} to write
\[
H^{i}_{\mathrm{et}}(X_{\CC}, \QQ_{\ell}) \cong
H^{i}_{\mathrm{et}}(Y_{\CC}, \QQ_{\ell})^G, \qquad
H^{i}(X_{\CC}^{\an}, \QQ_{\ell}) \cong
H^{i}(Y_{\CC}^{\an}, \QQ_{\ell})^G.
\]
\end{proof}

\begin{lemma} \label{L:global quotient}
There exist a smooth projective scheme $\overline{Y}$ over $\Spec(\ZZ_p)$, a relative normal crossings divisor $Z$ on $\overline{Y}$, and a finite group $G$ acting on both $\overline{Y}$ and $Z$ such that for $Y = \overline{Y} - Z$,
the stack-theoretic quotient $[Y/G]$ is isomorphic to $M_{g,n,\ZZ_p}$.
\end{lemma}
\begin{proof}
This is a consequence of the construction of \cite[\S 7.5]{ACV}, in which a suitable $\overline{Y}$ is realized as the moduli space of $n$-pointed genus $g$ curves with a certain \emph{nonabelian level structure}, i.e., a suitable finite Galois cover with fixed Galois group $H$. Note that the group $H$ has exponent equal to the product of two arbitrary primes, and so may be forced to be coprime to $p$; this ensures that $H$-covers are tamely ramified, which allows the construction to go through over $\Spec(\ZZ_p)$. (By contrast, the group $G$ may have order divisible by $p$.)
\end{proof}

Put
\[
R_{n,\ell}:= R_n \otimes_{\QQ} \QQ_\ell = \QQ_\ell[\psi_1,\dots,\psi_n,\kappa_1,\kappa_2,\dots],
\]
again graded by $\deg(\psi_i)=2$ and $\deg(\kappa_j)=2j$. Equip $R_{n,\ell}$ with a $\QQ_\ell$-linear endomorphism $\Frob$ as follows:
 \[
\begin{cases}
 \Frob \psi_i &= q \psi_i\\
 \Frob \kappa_j &= q^{j} \kappa_j.
 \end{cases}
\]
Let $R^i_{n,\ell}$ denote the $i$th graded piece of the ring.  For each $g,n$, we have a homomorphism of graded rings (with $\Frob$ action)
\begin{equation}\label{E:taut}
R_{n,\ell}^i \to H^{*}_{\mathrm{et}}(M_{g,n,\overline{\FF}_q}, \QQ_{\ell})
\end{equation}
again factoring through the Chow ring.

\begin{theorem}\label{T:etale}
For $0\leq i \leq \frac{2g-2}{3}$, the homomorphism in Equation~\eqref{E:taut} gives an isomorphism
of Frobenius modules
$$
 R_{n,\ell}^i \cong H^{i}_{\mathrm{et}}(M_{g,n,\overline{\FF}_q}, \QQ_{\ell}).
$$
\end{theorem}
\begin{proof}
Let $0\leq i \leq \frac{2g-2}{3}$.
Since the tautological classes arise from the Chow ring, they are of Tate type, so the map~\eqref{E:taut} is Frobenius-equivariant.
By Lemma~\ref{L:comparison of cohomology} and Lemma~\ref{L:global quotient},
given the choice of an embedding of $\overline{\QQ}_p$ into $\CC$,
there is a chain of functorial isomorphisms
\begin{equation}\label{E:taut-isos}
 R_{n,\ell}^i \to H^{i}_{\mathrm{et}}(M_{g,n,\overline{\FF}_q}, \QQ_{\ell}) \cong H^{i}_{\mathrm{et}}(M_{g,n,\CC}, \QQ_{\ell}) \cong H^{i}(M_{g,n,\CC}^{\an}, \QQ_{\ell}).
\end{equation}
Each step in the formation of the tautological classes involves either pushing forward or pulling back cohomology classes, or formation of Chern classes (which by \cite[Theorem 10.3]{Milne:etaleBook} are characterized entirely by certain maps on cohomology). Since each map in Equation~\eqref{E:taut-isos} is functorial, the tautological classes thus map to tautological classes. The composition is thus the isomorphism obtained from Theorem \ref{T:stab-taut-class} by extending scalars from $\QQ$ to $\QQ_{\ell}$; in particular, it does not depend on the embedding of $\overline{\QQ}_p$ into $\CC$. This means that in \eqref{E:taut-isos}, the composition and all but one of the maps are isomorphisms, so the remaining one is also an isomorphism and the claim follows.
\end{proof}

\begin{corollary}
For $0 \leq i \leq \frac{2g-2}{3}$, the following is true.
\begin{enumerate}
\def\theenumi{\alph{enumi}}
\item
If $i$ is odd, then $H^{2 d_{g,n} -i}_{c,\mathrm{et}}(M_{g,n,\overline{\FF}_q}, \QQ_\ell) = 0$.
\item
If $i$ is even, then
$H^{2 d_{g,n} -i}_{c,\mathrm{et}}(M_{g,n,\overline{\FF}_q}, \QQ_\ell)$
has $\QQ_\ell$-dimension equal to that of $R^i_{n,\ell}$, and $\Frob$ acts on it by multiplication by $q^{d_{g,n} - i/2}$.
\end{enumerate}
\end{corollary}
\begin{proof}
Since $M_{g,n,\overline{\FF}_q}$ is smooth, we may apply Poincar\'e duality for \'etale cohomology to deduce the claim from Theorem~\ref{T:etale}. (As in the proof of Theorem \ref{T:etale}, we may deduce duality for $M_{g,n,\overline{\FF}_q}$ from duality for smooth schemes via the Hochschild-Serre spectral sequence.)
\end{proof}

In our Grothendieck-Lefschetz trace computation, we will handle different parts of the cohomology of $M_{g,n,\overline{\FF}_q}$ in different ways.  We thus define
\begin{align*}
  \mathbf{T}^\stable_{g,n,q} &:= \sum_{0 \le i \le \floor{\frac{2g-2}3}}
  (-1)^i \Tr( \Frob, H^{2 d_{g,n}-i}_\cet(M_{g,n,\overline{\FF}_q},\QQ_\ell)) \\
  \mathbf{T}^\nonstable_{g,n,q} &:=
 \sum_{\floor{\frac{2g-2}3}< i \le 2 d_{g,n}} (-1)^i \Tr( \Frob, R^{2 d_{g,n} -i}_\cet(M_{g,n,\overline{\FF}_q},\QQ_\ell))\\
\mathbf{N}_{g,n,q} &:= \sum_{\floor{\frac{2g-2}3} < i \le 2 d_{g,n}} (-1)^i
 \Tr (\Frob, B_\cet^{2d_{g,n}-i}(M_{g,n,\overline{\FF}_q},\QQ_\ell)).
\end{align*}
Note that, since these account for all of the cohomology of $M_{g,n,\overline{\FF}_q}$, we have
\begin{equation}\label{eqn:MgnFq}
\#|M_{g,n}(\FF_q)| = \mathbf{T}^\stable_{g,n,q}+ \mathbf{T}^\nonstable_{g,n,q}+\mathbf{N}_{g,n,q}.
\end{equation}

\section{Heuristic~\ref{heur:tautological} yields Conjecture~\ref{conj:poisson1}}
\label{S:heuristic implies conjecture}

In this section, we prove Theorem~\ref{T:getconj}. We first note that
Heuristic~\ref{heur:tautological} is equivalent to the assertion that
\begin{equation}
\label{eqn:Ngn}
\lim_{g \rightarrow \infty}  q^{-d_{g,n}} \mathbf{N}_{g,n,q} = 0.
\end{equation}
Thus, to prove Theorem~\ref{T:getconj}, we need to understand the limiting behavior of $\mathbf{T}^\stable_{g,n,q}$ and $\mathbf{T}^\nonstable_{g,n,q}$.

Let $R_{n}$ be the tautological ring as defined in Section~\ref{S:stability}.  Note that the Hilbert series (or Poincar\'e series) $HS_{R_{n}}(z) :=\sum_{i=0}^\infty \dim R_{n}^{2i} \cdot z^{2i}$ may be rewritten as
\[
HS_{R_{n}}(z) = \prod_{i=1}^n \frac{1}{1-z^2} \prod_{j=1}^{\infty} \frac{1}{1-z^{2j}}.
\]

\begin{lemma}
\label{L:Tasymptotic}
We have the following:
\begin{enumerate}
\def\theenumi{\alph{enumi}}
\item $\lim_{g\rightarrow\infty}    q^{-d_{g,n}} \mathbf{T}^\stable_{g,n,q} =
    HS_{R_n}(q^{-1/2})$;
\item $\lim_{g\rightarrow\infty}    q^{-d_{g,n}} \mathbf{T}^\nonstable_{g,n,q} = 0$.
\end{enumerate}
\end{lemma}
\begin{proof}

For the first statement, we compute:
\begin{align*}
\lim_{g\to \infty} q^{-d_{g,n}}\mathbf T^\stable_{g,n,q}&=
\lim_{g\to \infty} q^{-d_{g,n}} \sum_{i=0}^{\lfloor \frac{2g-2}{3}\rfloor} (-1)^i \Tr(\Frob, R^{2d_{g,n}-i}_{c,\mathrm{et}}(M_{g,n,\overline{\FF}_q}, \mathbb{Q}_{\ell})) \\
&=\lim_{g\to \infty} \sum_{j=0}^{\lfloor \frac{g-1}{3}\rfloor} q^{ - j}\cdot \dim R_{n}^{2j}\\
&=\sum_{j=0}^{\infty} q^{- j}\cdot \dim R_{n}^{2j}.\\
\intertext{Using the Hilbert series of $R_{n}$, we may then rewrite the above sum as}
&= HS_{R_{n}}(q^{-1/2})\\
&= \prod_{i=1}^n \frac{1}{1-q^{-1}} \prod_{j=1}^{\infty} \frac{1}{1-q^{-j}}.
\end{align*}
Note that we use the fact (from Theorem~\ref{T:etale}) that for $0\leq i \leq \frac{2g-2}{3}$, we have
$$
 R_{n,\ell}^i = R_{n}^i \otimes_{\QQ} \QQ_{\ell} \cong R^{2d_{g,n}-i}_{c,\mathrm{et}}(M_{g,n,\overline{\FF}_q}, \mathbb{Q}_{\ell}).
$$

For the second part, we let $P(z)$ be the generating function for the partition numbers $p(j)$, and let $Q_n(z) := \sum \binom{n+j-1}{j}z^j$ be the generating function whose $j$th coefficient is the number of multisets of size $j$ on $n$ elements.  Then
\[
HS_{R_{n}}(z) = Q_n(z^2)P(z^2).
\]
In particular
\begin{equation}
\label{eqsubexp}
\dim R^{2i}_n = \sum_{j=0}^i \binom{n+j-i}{j-1}p(i-j)\leq \exp(c_n\sqrt{i}).
\end{equation}
Since $R^*_{c,\mathrm{et}}(M_{g,n,\overline{\FF}_q}, \QQ_{\ell})$ is defined in terms of the image of a map from $R_n$ to cohomology,  we further obtain
\begin{equation}
\label{eqsubexp2}
\dim R^{2d_{g,n}-2i}_\cet(M_{g,n,\overline{\FF}_q},\mathbb Q_{\ell}) \leq \exp(c_n\sqrt{i}).
\end{equation}
Of course, when $i$ is odd this group is zero-dimensional.

We compute
\begin{align*}
\lim_{g\rightarrow\infty}    q^{-d_{g,n}} \mathbf{T}^\nonstable_{g,n,q} &= \lim_{g\rightarrow\infty}q^{-d_{g,n}}\sum_{i=\floor{\frac{2g-2}3}+1}^{2d_{g,n}} (-1)^i \Tr( \Frob, R^{2d_{g,n}-i}_\cet(M_{g,n,\overline{\FF}_q},\QQ_\ell))\\
 &\le \lim_{g\rightarrow\infty}\sum_{i=\floor{\frac{2g-2}3}+1}^{2d_{g,n}} (-1)^iq^{-(i/2)} \dim R^{2i}_n\\
&\le \lim_{g\rightarrow\infty}\sum_{i=\floor{\frac{2g-2}3}+1}^{2d_{g,n}} (-1)^iq^{-(i/2)}\exp(c_n\sqrt{i})\\
&=0.
\end{align*}
\end{proof}

\begin{proof}[Proof of Theorem~\ref{T:getconj}]
By combining Equations~\eqref{eqn:MgnFq} and \eqref{eqn:Ngn} and Lemma~\ref{L:Tasymptotic} we get:
\begin{align*}
\lim_{g\to \infty} \frac{\#|M_{g,n}(\mathbb F_q)|}{\#|M_{g}(\mathbb F_q)|} &=
\lim_{g\rightarrow\infty}
\frac{\mathbf{T}^\stable_{g,n,q}+ \mathbf{T}^\nonstable_{g,n,q}+\mathbf{N}_{g,n,q}}{\mathbf{T}^\stable_{g,0,q}+ \mathbf{T}^\nonstable_{g,0,q}+\mathbf{N}_{g,0,q}}\\
&=q^n\frac{HS_{R_n}(q^{-1/2})+ 0+0}{HS_{R}(q^{-1/2})+0+0}\\
&=q^n \prod_{i=1}^n \frac{1}{1-q^{-1}}\\
&=\lambda^n.
\end{align*}
\end{proof}

\section{Proof of Theorem~\ref{T:weaker}}
\label{S:weaker}
In contrast to the previous sections, where $q$ was fixed, in this section we consider a case where $q$ and $g$ both go to infinity.  We show that, so long as $q$ goes to infinity much faster than $g$, then we obtain an unconditional version of Conjecture~\ref{conj:poisson1}.

The following lemma is the key result for this section, as it essentially shows that if $q\gg g$, then the main terms in the Grothendieck-Lefschetz trace computation will come from the stable cohomology range.
\begin{lemma}\label{L:constant K}
For any $K>144$ and any nonnegative integer $n$,
there exists a constant $K' = K'(n) > 0$
such that if $g > K'(n+1)$ and $q > g^K$, then
\[
\left| \mathbf T^\nonstable_{g,n,q}+ \mathbf N_{g,n,q} \right| < q^{d_{g,n}-g/6}.
\]
\end{lemma}

\begin{proof}
We bound the total cohomology of $M_{g,n,\overline{\FF}_q}$ by
$$
\sum_{i} \dim H^{i}_{c,\mathrm{et}}(M_{g,n,\overline{\FF}_q}, \mathbb{Q}_{\ell}) \leq (2+2g)^n (12g)!.
$$
For $n=0$, see \cite[Lemma 5.1]{deJongKatz}.  For general $n$, the bound follows from $n=0$ by iteratively applying the Serre spectral sequence for $M_{g,i+1,\overline{\FF}_q}$ over $M_{g,i, \overline{\FF}_q}$.

In addition, we note that each cohomology group which arises in the calculation of $\mathbf T^\nonstable_{g,n,q}$ and $\mathbf N_{g,n,q}$ is mixed of weight less than $2d_{g,n}-\floor{\frac{2g-2}{3}}.$  We thus have
\[
\left| \mathbf T^\nonstable_{g,n,q}+ \mathbf N_{g,n,q} \right| < q^{d_{g,n}- \floor{\frac{g-1}{3}}}(2g+2)^n(12g)!.
\]
To ensure that this is at most $q^{d_{g,n}-g/6}$, it suffices to take $q$ satisfying
\[
\left(\floor{\frac{g-1}3} - \frac{g}{6}\right)\log(q) > n\log(2g+2) + 12g \log(12g),
\]
which would in turn follow from
\[
\frac{1}{2}\left(\floor{\frac{g-1}3} - \frac{g}{6}\right)\log(q) > \max\{n\log(2g+2), 12g \log(12g)\}.
\]
Since $K > 144$, for any sufficiently small $\epsilon > 0$ we can choose $K'$ such that
for $g > K'(n+1)$ and $q > g^K$,
\[
\left(\floor{\frac{g-1}3} - \frac{g}{6}\right)\log(q) >
(1-\epsilon) \frac{g}{6}
\]
and
\[
\log(q) > \frac{144}{1-\epsilon} \log(12g).
\]
This proves the claim.
\end{proof}

\begin{proof}[Proof of Theorem~\ref{T:weaker}]
We combine Equation \eqref{eqn:MgnFq} and Lemmas~\ref{L:Tasymptotic} and ~\ref{L:constant K} to compute
\begin{align*}
\lim_{g\to \infty} \frac{\#|M_{g,n}(\mathbb F_q)|}{\#|M_{g}(\mathbb F_q)|} &=
\lim_{g\rightarrow\infty}
\frac{\mathbf{T}^\stable_{g,n,q}+ \mathbf{T}^\nonstable_{g,n,q}+\mathbf{N}_{g,n,q}}{\mathbf{T}^\stable_{g,0,q}+ \mathbf{T}^\nonstable_{g,0,q}+\mathbf{N}_{g,0,q}}\\
&=q^n\frac{HS_{R_n}(q^{-1/2})+ O(q^{-g/6})}{HS_{R}(q^{-1/2})+O(q^{-g/6})}\\
&=\lambda^n+O(q^{-g/6}).
\end{align*}
\end{proof}

\section{Connections to random matrix models}
\label{S:random matrix models}

Since much previous intuition about the behavior of random curves has come from the world of random matrix models, we would like to close with an invitation to the random matrix theory community to come up with evidence in favor of or opposed to Conjecture~\ref{conj:poisson1}. Let us say a few words about how this might be possible.

When $q$ is large compared to $g$, one typically models the behavior of the zeta function of a random curve $C$ of genus $g$ over $\FF_q$ by positing that the normalized characteristic polynomial of Frobenius behaves like that of a random matrix $M$ in the unitary symplectic group $\USp(2g)$. Equivalently, the sequence of point counts $\{\#C(\FF_{q^n})\}_{n=1}^\infty$ has the same distribution as
$\{q^n+1-q^{n/2} \Tr(M^n)\}_{n=1}^\infty$.

This model fails to apply in the case of fixed $q$ for three different reasons.
\begin{itemize}
\item
\textbf{Discreteness:}
For $n$ a positive integer,
because $\#C(\FF_{q^n}) \in \ZZ$,
one must insist that $\Tr(M^n) \in q^{-n/2} \ZZ$.
\item
\textbf{Positivity:}
Because $\#C(\FF_q) \geq 0$, one must insist that $q+1-q^{1/2} \Tr(M) \geq 0$.
\item
\textbf{More positivity}:
For $n_1, n_2$ two positive integers,
because $\#C(\FF_{q^{n_1n_2}}) \geq \#C(\FF_{q^{n_1}})$, one must insist that
$q^{n_1n_2} +1 - q^{n_1n_2/2}\Tr(M^{n_1n_2}) \geq
q^{n_1}+1-q^{n_1/2} \Tr(M^{n_1})$.
\end{itemize}

It seems unlikely that the statistics for such restricted random matrices can be computed in closed form, even in the limit as $g \to \infty$. However, it may be feasible to make numerical experiments for particular values of $q$ and $g$ to see how they compare to the predictions made by Conjecture~\ref{conj:poisson1}.

\bibliographystyle{hunsrt}
\bibliography{random}
\end{document}